\documentclass[11pt]{amsart}

\usepackage{amssymb,amsmath}

\usepackage{amsthm}
\usepackage{enumerate}
\usepackage[all,2cell]{xy}
\UseAllTwocells

\newtheorem{theorem}{Theorem}

\newtheorem{lemma}{Lemma}

\newcommand{\ctg}[1]{\hbox{\bf #1}}

\newcommand{\Sets}{\ctg{Sets}}

\newcommand{\nglob}[1]{\ctg{Glob}_{#1}}
\newcommand{\oglob}{\nglob{\omega}}

\newcommand{\ncat}[1]{\ctg{Cat}_{#1}}
\newcommand{\ocat}{\ncat{\omega}}
\newcommand{\npol}[1]{\ctg{Pol}_{#1}}
\newcommand{\opol}{\npol{\omega}}

\newcommand{\doubl}[2]{\ar@<2pt>[l]^{#2}\ar@<-2pt>[l]_{#1}}

\newcommand{\doubr}[2]{\ar@<2pt>[r]^{#2}\ar@<-2pt>[r]_{#1}}

\newcommand \free[1] {#1^{\ast}}

\newcommand \unit[1]{1_{#1}}

\newcommand{\adjd}[2]{\ar@<1ex>[d]^{#1}_{\dashv}\ar@<-1ex>@{<-}[d]_{#2}}
\newcommand{\adjr}[2]{\ar@<1ex>[r]^{#1}_{\top}\ar@<-1ex>@{<-}[r]_{#2}}

\newcommand{\UU}{U} % oubli cat-->glob
\newcommand{\WW}{V} % adjoint à droite cat --> pol
\newcommand{\GG}{G} % foncteur pol --> glob tq U=GW

\newcommand \sce[1] {s_{#1}}
\newcommand \tge[1] {t_{#1}}

\newcommand \trip[3] {(#1,#2,#3)}

\newcommand \NN {\mathbb N} % natural numbers

\newcommand \counit[1]{\epsilon^{#1}}

\newcommand \inc{j} % canonical inclusion

\newcommand \lift[1] {\overline{#1} }
\title{Strict $\omega$-categories are monadic over polygraphs}
\author{Fran\c cois M\'etayer}
\address{Universit\'e Paris Ouest Nanterre La D\'efense, IRIF,
          UMR 8243 CNRS, Univ Paris Diderot, Sorbonne Paris Cit\'e,
          F-75205 Paris, France}
\email{metayer@pps.univ-paris-diderot.fr}
\thanks{Supported by Cathre project, ANR-13-BS02-0005}

\begin{document}

\begin{abstract}
  We give a direct proof that the category of strict
  $\omega$-categories is monadic over the category of polygraphs.
\end{abstract}

\maketitle

\section*{Introduction}

This short note presents a proof of monadicity for the adjunction
between the category $\ocat$
of strict $\omega$-categories and the category $\opol$ of
polygraphs (or computads, as first introduced by Street
in~\cite{street:limicf}). Here we follow the presentation and
terminology of~\cite{burroni:highdw,metayer:respol}. The reader may
consult~\cite{metayer:cofohc} for a detailed description of the
categories and functors referred to in this particular case,
or~\cite{batanin:comfmg} for a broader perspective including
generalized ``$A$-computads'' for a monad $A$  on globular sets. 
The latter paper rightly asserts the monadicity
theorem, but some parts of the proof rely on the fact that the
category of $A$-computads is a presheaf category, which is precisely
not true in the present case, where $A$ is the monad of strict
$\omega$-categories~\cite{makkaizawadowski:tconcc,cheng:dircat}.
Since then, the status
of monadicity for $\ocat$ has remained somewhat unclear (see e.g the entry
``computad'' on the $n$Lab~\cite{nlab:comput}).  Our proof is
based on the same ideas as developed in~\cite{batanin:comfmg}, except that
we avoid the presheaf argument and establish instead a lifting
result~(Lemma~\ref{lemma:globlifting}), possibly of independent interest.

As for notations, whenever a functor $F$ is a right-adjoint, we denote
its left-adjoint by $\free F$. Let us finally mention a small point
about terminology. Given a functor $F:\ctg{A}\to \ctg{B}$, with
left-adjoint $\free F$, and $T=F\free F$ the associated monad
on $\ctg{B}$, there is a comparison functor $K$ from
$\ctg{A}$ to the category $\ctg{B}^T$ of $T$-algebras: we call $F$
{\em monadic} if $K$ is an equivalence of categories, and {\em
  strictly monadic} if $K$ is an isomorphism. We refer
to~\cite[VI.7]{maclane:catfwm} for corresponding variants of Beck's
monadicity criterion.

\section{Three adjunctions}

In this section, we briefly describe three pairs of adjoint functors between categories $\oglob$ of globular sets,
$\ocat$ of strict $\omega$-categories and $\opol$ of
polygraphs. 
\paragraph{}
As $\omega$-categories are globular sets with extra structure, there
is an obvious forgetful functor
\begin{displaymath}
  \UU:\ocat\to \oglob
\end{displaymath}
This functor $\UU$ has a left-adjoint $\free\UU$ taking a globular set $X$ to the
$\omega$-category $\free\UU X$ it generates. Moreover this adjunction
is strictly monadic.
\paragraph{}
A second adjunction involves functors
$\WW,\free\WW$ between $\ocat$ and $\opol$. Unlike $\UU$, the right
adjoint  $\WW$ is not quite obvious. 
Thus, let $C$ be an $\omega$-category, 
the polygraph $P=\WW(C)$ is defined by induction,
together with a morphism $\counit{C}:\free\WW (P)\to C$:
\begin{itemize}
\item For $n=0$, $P_0=C_0$ and $\counit{C}_0$ is the identity.
\item Suppose $n>0$, and $P$, $\counit{C}$ have been defined up to
  dimension $n{-}1$. The set of $n$-generators of $P$ is then the set
  $P_n$ of triples $p=\trip{z}{x}{y}$ where $z\in C_n$, $x$, $y$ are
  parallel cells in $\free{P}_{n-1}$ and
  $z:\counit{C}_{n-1}(x)\to\counit{C}_{n-1}(y)$. The source and target
  of $p$ in $\free{P}_{n-1}$ are $x=\sce{n-1}(p)$ and $y=\tge{n-1}(p)$
  respectively, and $\counit{C}_{n}(p)=z$. By the universal property
  of polygraphs, $\counit{D}_{n}$ extends uniquely to a map from
  $\free{P}_n$ to $C_n$ preserving compositions and
  identities. Functoriality of $\WW$ is immediate and  $\WW$ is in fact right-adjoint to
  $\free\WW$ (see~\cite{batanin:comfmg,metayer:respol}). 
\end{itemize}
Note that
  \begin{displaymath}
    \counit{C}:\free\WW\WW(C)\to C
  \end{displaymath}
is the counit of this adjunction and determines the
standard polygraphic resolution of $C$.
\paragraph{}
We finally describe a functor
\begin{displaymath}
  \GG:\opol\to\oglob
\end{displaymath}
Let $P$ be a polygraph. Let us denote by $\inc_n:P_n\to \free P_n$ the
canonical inclusion of the set  of $n$-generators of $P$ into the set of
$n$-cells of $\free P=\free\WW(P)$. We define the globular set
$X=\GG(P)$ dimensionwise, so that for each $n\in\NN$, $X_n\subset
P_n$:
\begin{itemize}
\item For $n=0$, $X_0=P_0$.
\item Let $n>0$ and suppose we have defined $X_k\subset P_k$ for all
  $k<n$, together with source and target maps building an
  $n{-}1$-globular set. Let $X_n\subset P_n$ be the set of
  $n$-generators $a$ of $P$ such that $\sce{n-1}(a)$ and
  $\tge{n-1}(a)$ belong to $\inc_{n-1}(X_{n-1})$ and define source and
  target maps $\sce{n-1}^X,\tge{n-1}^X:X_n\to X_{n-1}$ as the unique
  maps such that $\inc_{n-1}\sce{n-1}^X(a)=\sce{n-1}(a)$ and
  $\inc_{n-1}\tge{n-1}^X(a)=\tge{n-1}(a) $ for each $a\in X_n$. This
  extends $X$ to an $n$-globular set.
  \begin{equation}\label{diag:gdef}
    \xymatrix{X_n \ar@<2pt>[d]^{\tge{n-1}^X}\ar@<-2pt>[d]_{\sce{n-1}^X}\ar@{^{(}->}[r]& P_n \ar@<2pt>[rd]^{\tge{n-1}}\ar@<-2pt>[rd]_{\sce{n-1}}& \\
                      X_{n-1} \ar@{^{(}->}[r]& P_{n-1}\ar[r]_{\inc_{n-1}} & \free P_{n-1}}
  \end{equation}
 \end{itemize}
The previous construction is clearly functorial and defines the
required functor $\GG$. Remark that $\GG$ admits a left adjoint
$\free\GG:\oglob\to\opol$ which takes the globular set $X$ to a
polygraph $P$ such that $P_n=X_n$, in other words $\free\GG$
defines a
natural inclusion of $\oglob$ into $\opol$.

Note that $\GG$ forgets all generators of $P$
that are not ``hereditary  globular'', so that for instance $\GG(P)$
may have no cells at all  beyond dimension $1$. However, the following
result shows that the functor $\GG$ is not always trivial.

\begin{lemma}\label{lemma:ufactor}
  There is a natural isomorphism  $\phi:\GG\WW\to \UU$, that is,
  the following diagram commutes up to a natural isomorphism
  \begin{equation}
    \label{diag:ufactor}
    \xymatrix{\ocat \ar[r]^{\WW}\ar[d]_{\UU}& \opol\ar[ld]^{\GG}\\
                    \oglob &}
  \end{equation}
\end{lemma}
\begin{proof}
  Let $C$ be an $\omega$-category, and $X=\GG\WW(C)$. For each
  $n\in\NN$, let $\phi^C_n:X_n\to C_n$ be the composition of the
  following maps
  \begin{displaymath}
    \xymatrix{X_n \ar@{^{(}->}[r]& \WW(C)_n \ar[r]^(.4){\inc_n}& \free\WW\WW(C)_n \ar[r]^(.6){\counit{C}_n}& C_n}
  \end{displaymath}
As $\counit{C}$ is an $\omega$-morphism and~(\ref{diag:gdef})
commutes, the family $(\phi^C_n)_{n\in\NN}$ defines a globular morphism
$\phi^C:\GG\WW(C)\to \UU(C)$, natural in $C$. Thus we get a natural
transformation $\phi:\GG\WW\to\UU$.  
 
Let us now define $\chi^C_n:C_n\to X_n$ by induction on $n$ such that
$\phi^C_n\circ \chi^C_n=\unit{C_n}$: 
\begin{itemize}
\item For $n=0$, $X_0=C_0$ and $\phi^C_0=\unit{C_0}=\unit{X_0}$, so that
  $\chi^C_0:C_0\to X_0$ is also $\unit{C_0}=\unit{X_0}$.
\item Suppose $n>0$ and $\chi^C_k$ has been defined up to $k=n{-}1$,
  and let $z\in C_n$. Let $u=\sce{n-1}(z)$ and $v=\tge{n-1}(z)$ in $C_{n-1}$.
  By induction hypothesis, 
  $\chi^C_{n-1}(u)$ and $\chi^C_{n-1}(v)$ belong to $X_{n-1}$. Let
  $x=\inc_{n-1}\chi^C_{n-1}(u)$, $y=\inc_{n-1}\chi^C_{n-1}(v)$ in $\free\WW\WW(C)_{n-1}$
  and define $a=\chi^C_n(z)=\trip{z}{x}{y}$. By construction $a\in X_{n}$ and $\phi^C_n(a)=z$.
\end{itemize}
It remains to prove that $\phi^C_n$ is injective. We reason again by
induction on $n$:
\begin{itemize}
\item For $n=0$, $\phi^C_0$ is an identity, hence injective.
\item Suppose $n>0$ and $\phi^C_{n-1}$ injective. Let
  $a_i=\trip{z_i}{x_i}{y_i}\in X_n$ for $i=0,1$ such that
  $\phi^C_n(a_0)=\phi^C_n(a_1)$. Thus $z_0=z_1$. Also
  \begin{eqnarray*}
    \phi^C_{n-1}(\sce{n-1}^X(a_0)) & = & \sce{n-1}(\phi^C_n(a_0))\\
                                            & = &\sce{n-1}(\phi^C_n(a_1))\\
                                             & =& \phi^C_{n-1}(\sce{n-1}^X(a_1))
  \end{eqnarray*}
and because $\phi^C_{n-1}$ is injective, 
\begin{displaymath}
  \sce{n-1}^X(a_0)=\sce{n-1}^X(a_1)
\end{displaymath}
Now
\begin{eqnarray*}
  x_0 & = & \sce{n-1}(a_0)\\
        & = & \inc_{n-1}\sce{n-1}^X(a_0)\\
        & = & \inc_{n-1}\sce{n-1}^X(a_1)\\
        & = & \sce{n-1}(a_1)\\
        & = & x_1
\end{eqnarray*}
Likewise $y_0=y_1$, and we get $a_0=a_1$. Hence $\phi^C_n$ is
injective and we are done.
\end{itemize}
\end{proof}

\section{Lifting lemma}
The forgetful functor $\UU:\ocat\to\oglob$ is faithful, but clearly not
full. However, globular morphisms lift to $\omega$-morphisms in the
sense of the following result:
\begin{lemma}\label{lemma:globlifting}
   Let $C$, $D$ be $\omega$-categories and $\alpha:\UU(C)\to \UU(D)$ be a
   globular morphism. Then there is a unique morphism $\lift{\alpha}:\WW(C)\to \WW(D)$
in $\opol$ such that the following square commutes:
\begin{equation}
  \xymatrix @C=1.5cm{\UU\free\WW\WW(C)\ar[d]_{\UU(\counit{C})}\ar[r]^{\UU\free{\WW}(\lift{\alpha})} & \UU\free\WW\WW(D)\ar[d]^{\UU(\counit{D})}\\
\UU(C)  \ar[r]_{\alpha}&\UU(D)}
\label{diag:lift}
\end{equation}
 \end{lemma}
 
 \begin{proof}
   We build the required morphism $\lift{\alpha}:\WW(C)\to\WW(D)$ by induction
   on the dimension. Note that diagram~(\ref{diag:lift}) yields a
   diagram in $\Sets$ at any given dimension $n$. We may therefore drop
   the letter $\UU$ in the following computations. Also
   $\free{\lift{\alpha}}$ is short for $\free{\WW}(\lift{\alpha})$. 
   \begin{itemize}
   \item For $n=0$, we have $\WW(C)_0=C_0$, $\WW(D)_0=D_0$; also
     $\counit{C}_0$ and $\counit{D}_0$ are identities, so that
     $\lift{\alpha}_0=\alpha_0$ is the unique solution.
\item Suppose $n>0$ and we have defined $\lift{\alpha}$ satisfying the commutation
  condition, up to dimension $n{-}1$.
 Let $p=\trip{z}{x}{y}$ be an
  $n$-generator of $\WW(C)$. Suppose $\lift{\alpha}(p)=\trip{z'}{x'}{y'}$: the
  commutation condition implies $z'=\alpha(z)$,
  $x'=\free{\lift{\alpha}}_{n-1}(x)$ and $y'=\free{\lift{\alpha}}_{n-1}(x)$, so that $\lift{\alpha}$
  extends in at most one way to dimension $n$, and uniqueness
  holds. As for the existence, $x$, $y$ are parallel $(n{-}1)$-cells in
  $\free\WW\WW(C)_{n-1}$; by induction hypothesis, their images
  $x'=\free{\lift{\alpha}}_{n-1}(x)$ and $y'=\free{\lift{\alpha}}_{n-1}(x)$ are
  $(n{-}1)$-parallel cells in $\free\WW\WW(D)$. 
Again, by induction hypothesis,~(\ref{diag:lift}) commutes in dimension
$n{-}1$; also $\alpha$ is a globular map, hence
  \begin{eqnarray*}
    \sce{n-1}(z') & = & \sce{n-1}(\alpha_n(z))\\
                         & = & \alpha_{n-1}(\sce{n-1}(z))\\
                         & = & \alpha_{n-1}(\counit{C}_{n-1}(x))\\
                          & =&
                          \counit{D}_{n-1}(\free{\lift{\alpha}}_{n-1}(x))\\
                          & = & \counit{D}_{n-1}(x')
  \end{eqnarray*}
and likewise
\begin{displaymath}
  \tge{n-1}(z')= \counit{D}_{n-1}(y')
\end{displaymath}
Therefore $p'=\trip{z'}{x'}{y'}$ is an $n$-generator of $\WW(D)$. Also
$\sce{n-1}(p')=x'=\free{\lift{\alpha}}_{n-1}(x)=\free{\lift{\alpha}}_{n-1}(\sce{n-1}(p))$ and
$\tge{n-1}(p')=y'=\free{\lift{\alpha}}_{n-1}(y)=\free{\lift{\alpha}}_{n-1}(\tge{n-1}(p))$, so that
$\lift{\alpha}$ extends to a morphism in $\opol$ up to dimension $n$. Finally
the diagram~(\ref{diag:lift}) commutes in dimension $n$~: it is
sufficient to check this on generators, but
\begin{eqnarray*}
  \counit{D}_n\free{\lift{\alpha}}_n(p) & = & \counit{D}_n(p')\\
                                            & = & z'\\
                                            & = & \alpha_{n}(z)\\
                                            & = & \alpha_n\counit{C}_n(p)
\end{eqnarray*}
and we are done.
   \end{itemize}
 \end{proof}

\section{Monadicity}

We now turn to the main result.

\begin{theorem}\label{thm:monadic}
  The functor $\WW:\ocat\to\opol$ is monadic.
\end{theorem}

\begin{proof}
Recall that monadicity means here that  $\ocat$ is {\em equivalent} to 
the category of algebras of the monad $\WW\free\WW$ on
$\opol$. By using the corresponding version of Beck's criterion, this
amounts to show that (i)~$\WW$ reflects isomorphisms and (ii)~if
$f$, $g$ is a parallel pair of $\omega$-morphisms such that the pair $\WW(f)$,
$\WW(g)$ has a split coequalizer in $\opol$, then $f$, $g$ has a
coequalizer in $\ocat$, and $\WW$ preserves coequalizers of such pairs~
(see for instance \cite[VI.7, exercises~3 and~6]{maclane:catfwm}).
\paragraph{}
First, if $f:C\to D$ is an $\omega$-morphism such that $\WW(f)$ is an
isomorphism, then $\GG\WW(f)$ is an isomorphism in $\oglob$ and by
Lemma~\ref{lemma:ufactor}, $\UU(f)$ is an isomorphism. Now, $\UU$  reflects isomorphisms, hence $f$ is an isomorphism. Therefore
$\WW$ reflects isomorphisms as required.
\paragraph{}
Now, let $f,g:C\to D$ be a pair of
 $\omega$-morphisms and suppose
\begin{equation}
  \label{diag:coeqpol}
  \xymatrix{\WW(C) \doubr{\WW(g)}{\WW(f)}& \WW(D) \ar[r]_k\ar@/_1.5pc/[l]_b& P\ar@/_1.5pc/[l]_a}
\end{equation}
is a split coequalizer in $\opol$ where $k\circ a=\unit{P}$, $\WW
(f)\circ b=\unit{\WW(D)}$ and $\WW(g)\circ b=a\circ k$. By applying the
functor $\GG$ to~(\ref{diag:coeqpol}), we get a split coequalizer
in $\oglob$:
\begin{equation}
  \label{diag:coeqglob1}
  \xymatrix{\GG\WW(C) \doubr{\GG\WW(g)}{\GG\WW(f)}& \GG\WW(D)
    \ar[r]_{\GG(k)}\ar@/_1.5pc/[l]_{\GG(b)}& \GG(P)\ar@/_1.5pc/[l]_{\GG(a)}}
\end{equation}
Then, by using the natural isomorphism $\phi$ of
Lemma~\ref{lemma:ufactor}, we obtain the following diagram
\begin{equation}
  \label{diag:coeqglob2}
  \xymatrix{\GG\WW(C) \doubr{\GG\WW(g)}{\GG\WW(f)}\ar[d]_{\phi^C}& \GG\WW(D)
    \ar[rd]|{\GG(k)}\ar@/_1.5pc/[l]_{\GG(b)}\ar[d]_{\phi^D}&
    \\
      \UU(C)\doubr{\UU(g)}{\UU(f)} & \UU(D)\ar[r]^l\ar@/^1.5pc/[l]^{\beta}&\GG(P)\ar@/_1.5pc/[lu]_{\GG(a)}\ar@/^1.5pc/[l]^{\alpha}}
\end{equation}
where $\alpha=\phi^D\circ \GG(a)$, $l=\GG(k)\circ (\phi^D)^{-1}$ and
$\beta=\phi^{C}\circ\GG(b)\circ (\phi^D)^{-1}$. Therefore $l\circ
\alpha=\unit{\GG(P)}$, $\UU(f)\circ \beta=\unit{\UU(D)}$ and 
\begin{eqnarray*}
  \UU(g)\circ\beta & = & \UU(g)\circ \phi^{C}\circ\GG(b)\circ
  (\phi^D)^{-1}\\
                             & = & \phi^D\circ
                             \GG\WW(g)\circ\GG(b)\circ (\phi^D)^{-1}\\
                              & = & \phi^D\circ
                             \GG(a)\circ\GG(k)\circ (\phi^D)^{-1}\\
                             & = & \alpha\circ l
\end{eqnarray*}
and the bottom line of~(\ref{diag:coeqglob2}) is a split coequalizer
diagram in $\oglob$. Now the functor $\UU$ is strictly monadic, so that there is a unique
$\omega$-morphism $h:D\to E$ such that $\UU(E)=\GG(P)$ and $\UU(h)=l$
and moreover this unique morphism makes
\begin{equation}
  \label{diag:coeqcat}
  \xymatrix{C \doubr{g}{f}& D\ar[r]^h & E}
\end{equation}
a coequalizer diagram in $\ocat$. Note that, by construction,  $\UU(E)=\GG(P)$.

It remains to show that $\WW(h):\WW(D)\to\WW(E)$ is a coequalizer of
the pair $\WW(f)$, $\WW(g)$ in $\opol$. By applying
Lemma~\ref{lemma:globlifting} to $\alpha:\UU(E)\to \UU(D)$ and to
$\beta:\UU(D)\to\UU(C)$, we get unique morphisms
$\lift{\alpha}:\WW(E)\to\WW(D)$ and $\lift{\beta}:\WW(D)\to\WW(C)$
satisfying the required commutation condition. Consider the following
diagram:
\begin{equation}
  \label{diag:uniquelift}
  \xymatrix@C=1.5cm{\UU\free\WW\WW(E)\ar[d]_{\UU(\counit{E})}\ar[r]^{\UU\free{\WW}(\lift{\alpha})}
    & \UU\free\WW\WW(D)\ar[d]|{\UU(\counit{D})}
\ar[r]^{\UU\free{\WW}\WW(h)}& \UU\free\WW\WW(E)\ar[d]^{\UU(\counit{E})}\\
\UU(E)  \ar[r]_{\alpha}&\UU(D)\ar[r]_{\UU(h)} & \UU(E)}
\end{equation}
The left-hand square commutes by hypothesis, and the right-hand square
commutes by the naturality of $\counit{}$, whence the outer square
also commutes. As $\UU(h)\circ\alpha=\unit{\UU(E)}$, the uniqueness of
the lifting in Lemma~\ref{lemma:globlifting} implies that $\WW(h)\circ
\lift{\alpha}=\unit{\WW(E)}$. By the same uniqueness argument, we get
$\WW(f)\circ\lift{\beta}=\unit{\WW(D)}$ and $\WW(g)\circ
\lift{\beta}=\lift{\alpha}\circ\WW(h)$. Therefore the following
diagram is a split coequalizer in $\opol$
\begin{displaymath}
 \xymatrix{\WW(C) \doubr{\WW(g)}{\WW(f)}& \WW(D) \ar[r]_k\ar@/_1.5pc/[l]_{\lift{\beta}}& \WW(E)\ar@/_1.5pc/[l]_{\lift{\alpha}}} 
\end{displaymath}
and we are done.
\end{proof}

\section*{Acknowledgements}
Many thanks to Dimitri Ara and Albert Burroni for numerous helpful
conversations on the subject.


\begin{thebibliography}{M{\'e}t08}

\bibitem[Bat98]{batanin:comfmg}
Michael Batanin.
\newblock Computads for finitary monads on globular sets.
\newblock {\em Contemp. Math.}, 230:37--57, 1998.

\bibitem[Bur93]{burroni:highdw}
Albert Burroni.
\newblock Higher-dimensional word problems with applications to equational
  logic.
\newblock {\em Theoretical Computer Science}, 115:43--62, 1993.

\bibitem[Che13]{cheng:dircat}
Eugenia Cheng.
\newblock A direct proof that the category of $3$-computads is not cartesian
  closed.
\newblock {\em Cahiers de Topologie et G{\'e}om{\'e}trie Diff{\'e}rentielle
  Cat{\'e}goriques}, 54(1):3--12, 2013.

\bibitem[M{\'e}t03]{metayer:respol}
Fran{\c c}ois M{\'e}tayer.
\newblock Resolutions by polygraphs.
\newblock {\em Theory and Applications of Categories}, 11(7):148--184, 2003.
\newblock http://www.tac.mta.ca/tac/.

\bibitem[M{\'e}t08]{metayer:cofohc}
Fran{\c c}ois M{\'e}tayer.
\newblock Cofibrant objects among higher-dimensional categories.
\newblock {\em Homology, Homotopy and Applications}, 10(1):181--203, 2008.
\newblock http://intlpress.com/HHA/v10/n1/a7/.

\bibitem[ML71]{maclane:catfwm}
Saunders Mac~Lane.
\newblock {\em Categories for the Working Mathematician}.
\newblock Springer, 1971.

\bibitem[MZ08]{makkaizawadowski:tconcc}
Michael Makkai and Marek Zawadowski.
\newblock The category of $3$-computads is not cartesian closed.
\newblock {\em J. Pure Appl. Algebra}, 212(11):2543--2546, 2008.

\bibitem[nLa]{nlab:comput}
nLab.
\newblock ncatlab.org/nlab/show/computad.

\bibitem[Str76]{street:limicf}
Ross Street.
\newblock Limits indexed by category-valued 2-functors.
\newblock {\em Journal of Pure and Applied Algebra}, 8:149--181, 1976.

\end{thebibliography}
\end{document}